\documentclass[12pt,a4paper]{amsart}
\usepackage{amsfonts}
\usepackage{amssymb}
\usepackage{amsthm}
\usepackage{amsmath}
\usepackage{amscd}
\usepackage{t1enc}
\usepackage[mathscr]{eucal}
\usepackage{indentfirst}
\usepackage{graphicx}
\usepackage{graphics}
\usepackage{pict2e}
\usepackage{epic}
\numberwithin{equation}{section}
\usepackage[margin=2.9cm]{geometry}
\usepackage{epstopdf} 
\usepackage{hyperref}

\hypersetup{
    colorlinks = true,
    urlcolor   = blue,
    citecolor  = blue,
}

\theoremstyle{plain}
\newtheorem{Th}{Theorem}[section]
\newtheorem{Lem}[Th]{Lemma}

\newtheorem{Prop}[Th]{Proposition}

\theoremstyle{definition}
\newtheorem{Def}[Th]{Definition}

\newtheorem{Rem}[Th]{Remark}
\newtheorem{?}[Th]{Problem}

\begin{document}

\title{Bounded geodesic image theorem via bicorn curves}

\author{Xifeng Jin}

\address{Xifeng Jin\\
Department of Mathematics\\
University at Buffalo--SUNY\\
Buffalo, NY 14260-2900, USA\\
xifengji@buffalo.edu}

\date{\today}

\begin{abstract} 
We give a uniform bound of the bounded geodesic image theorem for the closed oriented surfaces. The proof utilizes the bicorn curves introduced by Przytycki and Sisto \cite{PS}. With the uniformly bounded Hausdorff distance of the bicorn paths and 1-slimness of the bicorn curve triangles, we are able to show the bound is 44 for both non-annular and annular subsurfaces. In a particular case when the distance between a geodesic and an essential boundary component of subsurface (or core if it is annular) is $\geq$ 18, then the bound can be as small as 3, which is comparable to the bound 4 in the motivating examples by Masur and Minsky \cite{MM2}, and is same as the bound given by Webb \cite{Webb2} for non-annular subsurfaces.
\end{abstract}
\maketitle

%%%%%%%%%%%%%%%%%%%%%%%%%%%%%%%%%%%%%%%%%%%%%%%%%%%%%%%%%%%%%%%%%%%%%%%%%%%%%%%%%%%%%%%%%%%%%%%%%%%%%%%%%%%%%%%%%%%%%%%%%%%%%%%%%%%%%%%%%%
\section{Introduction} 

Let $S = S_{g, b}$ be a compact oriented surface with genus $g$ and $b$ boundary components, the \emph{complexity} of the surface is $\xi(S)=3g+b-3$. A simple closed curve on the surface $S$ is \emph{essential} if it does not bound a disk or an annulus. A simple properly embedded arc on the surface $S$ is \emph{essential} if it does not cut off a disk.

%%%%%%%%%%%%%%%%%%%%%%%%%%%%%%%%%%%%%%%%%%%%%%
\subsection{Curve complex}

Suppose that $\xi(S) \geq 2$, the \emph{arc and curve graph} $\mathcal{AC}(S)$ is the graph whose vertices are the isotopy classes of essential properly embedded arcs and essential curves on the surface $S$, where the isotopy allows the endpoints of the arcs move around in the corresponding boundary components. Two vertices can be connected by an edge in $\mathcal{AC}(S)$ if they can be realized disjointly. The \emph{curve graph} $\mathcal{C}(S)$ is the full subgraph of $\mathcal{AC}(S)$ whose vertices are the isotopy classes of essential curves. The \emph{curve graph} is the 1-skeleton of the curve complex introduced by Harvey \cite{Harvey} with the aim to study the mapping class group.  $\mathcal{C}_0(S)$ and $\mathcal{AC}_0(S)$ are used to denote the vertices of $\mathcal{C}(S)$ and $\mathcal{AC}(S)$, respectively.

For the surfaces with $\xi(S) \leq 1$, we need to modify the definition slightly. If $\xi(S) = 1$, the surface $S$ is either a one-holed torus $S_{1,1}$ or a four-holed sphere $S_{0, 4}$. In both cases, two vertices are connected by an edge if some of their curve representatives intersect exactly once for $S_{1,1}$ or twice for $S_{0, 4}$. The resulting curve graphs are isomorphic to the Farey graph. If $\xi(S) = 0$, then $S = S_{1}$ is a torus or $S = S_{0,3}$ is a pair of pants. Similar to the once-punctured torus,  the curve graph of torus is isomorphic to the Farey graph if two curve representatives intersect exactly once are joined by an edge. The curve graph of a pair of pants is empty. If $S = S_{0,2}$ is an annulus, we define the vertices of $\mathcal{C}(S)$ to be the essential arcs up to isotopy fixing the boundary component pointwise, and two vertices can be joined by an edge if they can be realized with disjoint interiors. In this case, we let $\mathcal{C}(S) = \mathcal{AC}(S)$.

For any two vertices $x$ and $y$ in the $\mathcal{C}_0(S)$, the \emph{distance} $d_{\mathcal{C}(S)}(x,y)$ between $x$ and $y$
is the minimal number of edges in $\mathcal{C}(S)$ joining $x$ and $y$. A 
\emph{geodesic} in the curve graph $\mathcal{C}(S)$ is a sequence of vertices 
$\Gamma=(\gamma_i)_{i \in I}$ such that $d_{\mathcal{C}(S)}(\gamma_i,\gamma_j)=|i-j|$ for 
all $i, j$ in $I$. These notions can be defined on the arc and curve graph $\mathcal{AC}(S)$ in the same manner. A metric space is called \emph{$\delta$-hyperbolic} if for every geodesic triangle, each side lies in a $\delta$-neighborhood of the other two sides.  A seminal result of the curve graphs is its hyperbilicity proved by Masur and Minsky.

\begin{Th} \emph{(Masur-Minsky \cite{MM1} Theorem 1.1)}
\label{infinite diameter}
Let $S$ be a compact oriented surface with $\xi(S) \geq 1$, the 
curve graph $\mathcal{C}(S)$ is a $\delta$-hyperbolic metric space with infinite 
diameter for some $\delta$, where $\delta$ depends on the surface.   
\end{Th}

Alternative proofs were given by Bowditch \cite{B1} and Hamenst\"{a}dt \cite{Ha}. Moreover, the hyperbilicity constant $\delta$ can be chosen to be independent of the surface. The existence of such uniform constant has been proved independently by Aougab \cite{Aougab}, Bowditch \cite{B2}, Clay-Rafi-Schleimer \cite{CRS}, Hensel-Przytycki-Webb \cite{HPW}. Przytycki and Sisto \cite{PS} also proved it for the closed oriented surfaces.

%%%%%%%%%%%%%%%%%%%%%%%%%%%%%%%%%%%%%%%%%%%%%%
\subsection{Subsurface projection}

A subsurface $Y$ is an \emph{essential subsurface} of $X$, 
if $Y$ is a compact, connected, oriented, proper subsurface such that each component of $\partial Y - \partial X$ is essential in $X$. Suppose $\xi(Y) \geq 1$, we define a map $\pi_A:\mathcal{C}_0(X)\rightarrow 
\mathcal{P}(\mathcal{AC}_0(Y))$, where $\mathcal{P}(\mathcal{AC}_0(Y))$ is the power 
set of $\mathcal{AC}_0(Y)$. Take any $\alpha$ in 
$\mathcal{C}_0(X)$, then consider the representative of $\alpha$ such that it 
intersects $Y$ minimally, $\pi_A(\alpha)$ is the set of all isotopy classes of 
$\alpha \cap Y$ relative to the boundary of $Y$. $\pi_A(\alpha)$ is empty if $\alpha$ can be 
realized disjointly from $Y$. We say $\alpha$ \emph{cuts} $Y$ if $\alpha \cap 
Y \neq \varnothing$, and $\alpha$ \emph{misses} $Y$ if $\alpha \cap Y = \varnothing$.

There is a natural way to send the arcs back to the curves in $\mathcal{C}_0(Y)$. We can define $\pi_0: \mathcal{AC}_0(Y) \rightarrow \mathcal{C}_0(Y)$ as 
follows. If $\alpha$ is in the $Y$, then $\pi_A(\alpha) = \alpha$ in $\mathcal{AC}_0(Y)$ and $\pi_0( \pi_A(\alpha)) = \alpha$ in $\mathcal{C}_0(Y)$. Otherwise, 
$\pi_A(\alpha)=\{\alpha_1,\alpha_2,\cdots,\alpha_n\}$ is a collection of isotopy 
classes of essential properly embedded arcs in $Y$. The set 
$\pi_0(\pi_A(\alpha))$ is the isotopy classes of the essential components of 
$\partial N(\alpha_i\cup \partial Y)$ in $Y$, where
$N(\alpha_i\cup Y)$ is a regular neighborhood of $\alpha_i\cup \partial Y$ in 
$Y$. The composition 
$\pi_0\circ \pi_A = \pi_Y: \mathcal{C}_0(X) \rightarrow \mathcal{C}_0(Y)$ is called the \emph{subsurface projection}.

If $Y$ is an annulus, an alternative definition is needed. Let $X$ be endowed with a complete hyperbolic metric with finite area and $Y \subset X$ is an essential annulus. Let $p_Y: \widetilde{X}_Y \rightarrow X$ be the annular covering map such that $Y$ lifts homeomorphically to $Y' \subset \widetilde{X}_Y$ and $\widetilde{X}_Y \cong \text{interior}(Y)$. The compactification of $\widetilde{X}_Y$ with the hyperbolic metric induced from $X$ is a closed annulus and it is denoted by $\overline{X}_Y$. If a curve $\alpha$ cuts $Y$, the subsurface projection $\pi_Y(\alpha)$ is the set of arcs of the full preimage $\widetilde{\alpha} = p_Y^{-1}(\alpha)$ that connect the two boundary components of $\overline{X}_Y$. If a curve $\alpha$ misses $Y$, then $\pi_Y(\alpha) = \varnothing$. 
Suppose that $v$ is the core of $Y$, then we also use $\pi_v$ to denote the projection and $d_v$ to denote the distance.
For more details about the subsurface projection and bounded geodesic image theorem, see the paper of Masur and Minsky \cite{MM2}.

\begin{Th}\emph{(Bounded Geodesic Image Theorem, Masur-Minsky \cite{MM2} Theorem 3.1)}
\label{bounded}
Let $Y$ be a subsurface of $X$ with $\xi(Y)\neq 0$ and let $\Gamma=(\gamma_i)_{i\in I}$ be a geodesic in $\mathcal{C}(X)$. If each $\gamma_i$ cuts $Y$, then there is a constant $M$ depending only on the surface so that $d_Y(\Gamma)\leq M$. 
\end{Th}

The notation $d_Y(\Gamma): = \text{diam}_{\mathcal{C}(Y)} (\underset{\gamma_i \in \Gamma}{\cup} \pi_Y(\gamma_i))$ is used above.

\begin{Lem}\emph{(Masur-Minsky \cite{MM2} Lemma 2.2)}
\label{2-quasi}
Let $\xi(Y) \geq 1$, and $d_{\mathcal{AC}(Y)}(\alpha,\beta)\leq 1$ for $\alpha, \beta \in \mathcal{AC}(Y)$, then
$d_{\mathcal{C}(Y)}(\pi_0(\alpha),\pi_0(\beta))\leq 2$.  
\end{Lem}

By the Lemma \ref{2-quasi}, one only needs to consider the diameter of the projection $\pi_A:\mathcal{C}_0(X)\rightarrow \mathcal{P}(\mathcal{AC}_0(Y))$, because $\pi_0: \mathcal{AC}_0(Y) \rightarrow \mathcal{C}_0(Y)$ is 2-Lipschitz.  Moreover, the projection $\pi_A$ satisfies the 1-Lipschitz property for the surfaces with $\xi(X) \geq 2$.

\begin{Lem}\emph{(Webb \cite{Webb} Lemma 1.2)}
\label{arc bounds}
 Let $Y$ be an essential subsurface of $X$ and let $\gamma_1,\gamma_2$ be curves on $X$. Suppose that $\gamma_1$ cuts $Y$, $\gamma_2$ cuts $Y$ and $\gamma_1$ misses $\gamma_2$. Then $d_{\mathcal{AC}(Y)}(\pi_A(\gamma_1),\pi_A(\gamma_2))\leq1$.
\end{Lem}

Using the notation $d_{\mathcal{AC}(Y)} (\Gamma) := \text{diam}_{\mathcal{AC}(Y)}(\underset{\gamma_i \in \Gamma}{\cup} \pi_A(\gamma_i))$,  we restate the theorem for closed oriented surfaces.

\begin{Th}
\label{small bounds}
Suppose that $X$ is a closed oriented surface with  genus $\geq 2$. Let $Y$ be an essential subsurface of $X$ with $\xi(Y)\neq 0$ and let $\Gamma=(\gamma_i)_{i\in I}$ be a geodesic in $\mathcal{C}(X)$ such that each $\gamma_i$ cuts $Y$, then $d_{\mathcal{AC}(Y)}(\Gamma)\leq 44$ for both non-annular and annular subsurfaces $Y$. 
\end{Th}

The bounded geodesic image theorem was originally proved by Masur and Minsky \cite{MM2}.  With the uniform hyperbilicity of the curve graphs, Webb \cite{Webb} proved that the constant $M$ is independent of the surface. Later on, he gave explicit bounds in his dissertation \cite{Webb2} using the unicorn arcs, where the bounds is 62 for annular subsurfaces, and 50 for non-annular subsurfaces.  In \cite{Webb,Webb2}, Webb remarked that an unpublished proof of Chris Leininger combined with the work of Bowditch \cite{B2} also provided a uniform bound. Our proof relies on a uniformly bounded Hausdorff distance of the bicorn paths and 1-slimness of the bicorn curve triangles. The outline of the proof is similar to Webb's, while the proof is simple and the bound is smaller for the closed oriented surfaces. The approach can be generalized to the surfaces with boundary, possibly with larger bounds.

The paper is organized as follows. In the Section \ref{sec: example}, we briefly describle the motivating example by Masur and Minsky.
 We recall the definition of bicorn curves and bicorn paths between two curves in the Section \ref{sec: bicorn curves}. In the Section \ref{sec: proof}, we obtain a uniform bound for the bounded geodesic image theorem for closed oriented surfaces.

\section*{Acknowledgements}

The author would like to thank his advisor, Professor William W. Menasco, for alerting the author to the small bound in the motivating examples by Masur and  Minsky \cite{MM2}, and the encouragement of obtain small bounds for the bounded geodesic
image theorem. The author would also like to thank Professor Johanna Mangahas for reading a draft of this paper and offering helpful comments. Finally, the author is grateful to the support of the Dissertation Fellowship for Fall 2020 from the Department of Mathematics at the State University of New York at Buffalo.

%%%%%%%%%%%%%%%%%%%%%%%%%%%%%%%%%%%%%%%%%%%%%%%%%%%%%%%%%%%%%%%%%%%%%%%%%%%%%%%%%%%%%%%%%%%%%%%%%%%%%%%%%%%%%%%%%%%%%%%%%%%%%%%%%%%%%%%%%%

\section{A motivating example}
\label{sec: example}
\begin{figure}[ht]
\scalebox{.5}{\includegraphics[origin=c]{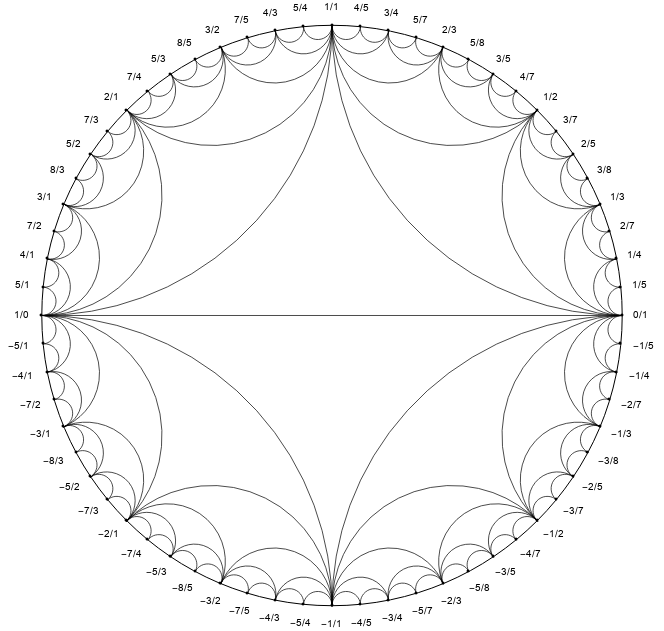}}
\caption{A typical Farey graph. Each vertex is labeled with a number in $\hat{\mathbb{Q}}$ that is a slope of the curve representative under a fixed homology basis. The figure was generated by Wolfram Mathematica, and the code snippet is availabe on the $\href{https://mathematica.stackexchange.com/questions/194838/how-can-i-plot-a-farey-diagram}{Stackexchange}$.}
\label{Farey}
\end{figure}

In this section, we will briefly describe a motivating example for the bounded geodesic image theorem from the Section 1.5 in \cite{MM2}.
The curve graph $\mathcal{C}(S_{1})$, $\mathcal{C}(S_{1,1})$ and $\mathcal{C}(S_{0,4})$ are isomorphic to the \emph{Farey graph}, as roughly illustrated in the Figure \ref{Farey}. It can be embedded in a unit disk, in which the rational numbers and the infinity  $\hat{\mathbb{Q}} = \mathbb{Q} \cup \{\infty\}$ are located on the boundary of the unit disk as the vertices of the graph.  Any two vertices $\frac{p_1}{q_1}$ and $\frac{p_2}{q_2}$ in lowest terms are connected by an edge if and only if the determinant $p_1 q_2 - p_2q_1 = \pm1$. Let $v$ be a vertex, then the \emph{link} of $v$, $\text{Link}(v)$, is a set of vertices that share an edge with $v$, which can be identified with the integers $\mathbb{Z}$. For example, if $v = \frac{0}{1}$, then the $\text{Link}(v) = \{ \cdots, -\frac{1}{2}, -\frac{1}{1}, \frac{1}{0}, \frac{1}{1}, \frac{1}{2}, \cdots \}$.

Let $u$ and $w$ be two vertices of a geodesic $h$ in the Farey graph, and a vertex $v$ of $ h$ follows $u$ and followed by $w$, the subsurface projection $d_v(u, w)$ defines a distance in the $\text{Link}(v)$, where $v$ is the core of an annular subsurface.  Let $h'$ be another geodesic with the same endpoints as $h$. If $d_v(u, w) \geq 5$, then $h'$ must contain the vertex $v$. In other words, if a geodesic does not contain $v$ (each vertex of the geodesic cuts $v$), then $d_v(u, w) \leq 4$ for any $u, w$ in the geodesic.

%%%%%%%%%%%%%%%%%%%%%%%%%%%%%%%%%%%%%%%%%%%%%%%%%%%%%%%%%%%%%%%%%%%%%%%%%%%%%%%%%%%%%%%%%%%%%%%%%%%%%%%%%%%%%%%%%%%%%%%%%%%%%%%%%%%%%%%%%%

\section{Bicorn curves}
\label{sec: bicorn curves}

Przytycki and Sisto \cite{PS} introduced the bicorn curves to give a simple proof of the hyperbilicity of curve graphs for the closed oriented surfaces $S_{g \geq 2}$. For more applications of the bicorn curves, see \cite{CJM,MS,Pho,Ra1,Ra2,Vokes}.

\begin{Def}[\textbf{Bicorn curves}]
Let $\alpha, \beta \subset S_{g \geq 2}$ be two essential simple closed curves that intersect minimally.  An essential simple closed curve $\gamma$ is a \emph{bicorn curve}
between $\alpha$ and $\beta$ if either $\gamma=\alpha$, $\gamma=\beta$, or $\gamma$
is represented by the union of an arc $a \subset \alpha$ and an arc $b \subset \beta$, which we call the $\alpha$-arc and the $\beta$-arc of $\gamma$, and $a$ only intersects $b$ at the endpoints.
If $\gamma = \alpha$, then its $\alpha$-arc is $\alpha$ and its $\beta$-arc is empty, similarly if $\gamma = \beta$, then its $\beta$-arc is $\beta$ and its $\alpha$-arc is empty.
\end{Def}

\begin{figure}[ht]
\scalebox{.30}{\includegraphics[origin=c]{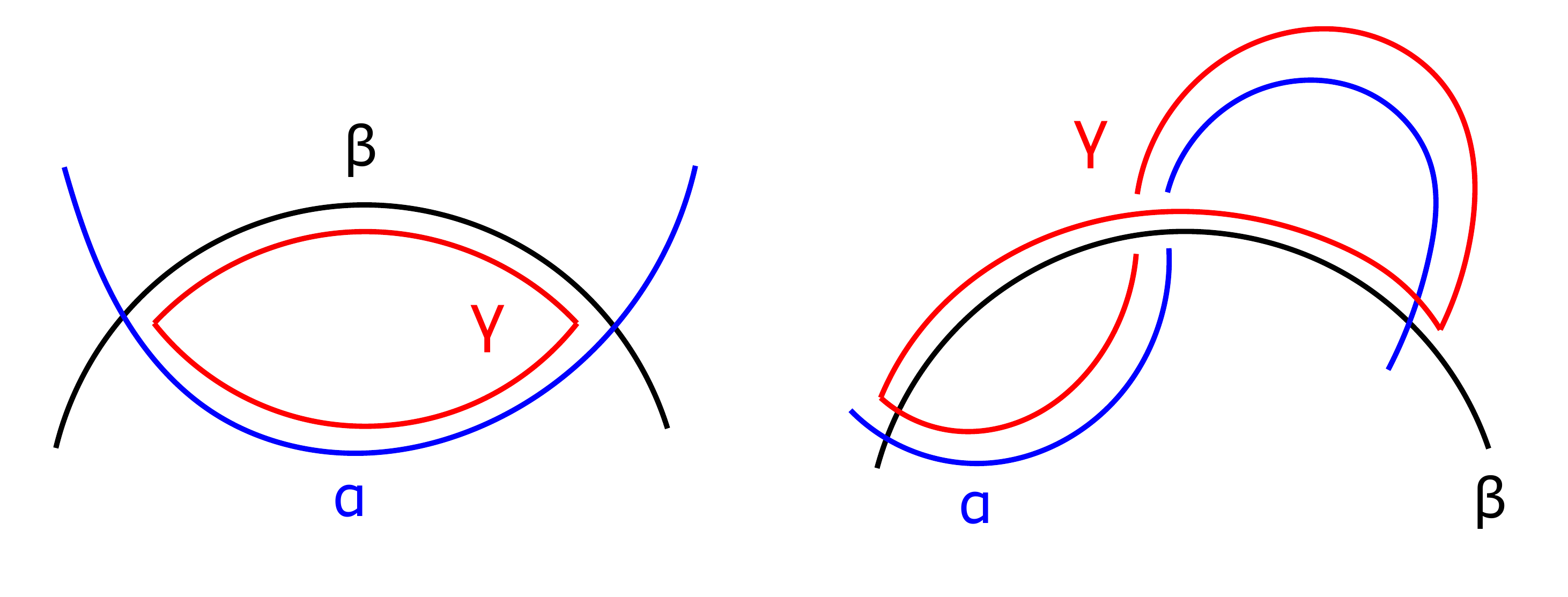}}
\caption{Two configurations of a bicorn curve $\gamma$ between $\alpha$ and $\beta$; The two intersection points have opposite orientations on the left and same orientation on the right. }
\label{bicorn}
\end{figure}

Based on the orientations of two intersection points, there are two configurations of the bicorn curves illustrated in the Figure \ref{bicorn}. If the surface is closed, any bicorn curves defined above are essential, as $\alpha$ and $\beta$ intersect minimally. The intersection number of $\alpha$ and $\beta$ is finite, so the number of bicorn curves is finite.

The collection of bicorn curves can be partially ordered. Two bicorn curves $\gamma < \gamma'$ if the $\beta$-arc of $\gamma'$ strictly contains the $\beta$-arc of $\gamma$.  Take a minimal subarc $b' \subset \beta$ that does not intersect $\alpha$ except for the endpoints. Denote the bicorn curve $\alpha_1=a' \cap b'$ as the union of the minimal subarc $b' \subset \beta$ and the subarc $a'$ of $\alpha$ determined by the endpoints. Then, $\alpha$ intersects $ \alpha_1$ at most once.  One can extend the minimal subarc $b'$ to the next intersection point with $a'$. The extended subarc $b''$ of $\beta$ intersects $a'$ on the endpoint, the bicorn curve is denoted as $\alpha_2=a''\cup b''$. See the Figure \ref{bicorn_extension}. As we can see, $\alpha_1$ intersects $\alpha_2$ at most once. 
 
 Next, we extend the subarc $b''$ to the minimal subarc $b'''$ such that $b'''$ intersects $a''$ right on the endpoint, the subarc of $a''$ with bounded by the new intersection point is denoted by $a'''$. The bicorn curve $\alpha_3=a'''\cup b'''$ intersects $\alpha_2$ at most once.

\begin{figure}[ht]
\scalebox{.40}{\includegraphics[origin=c, width=1.20 \textwidth,  height = 1.0 \textwidth]{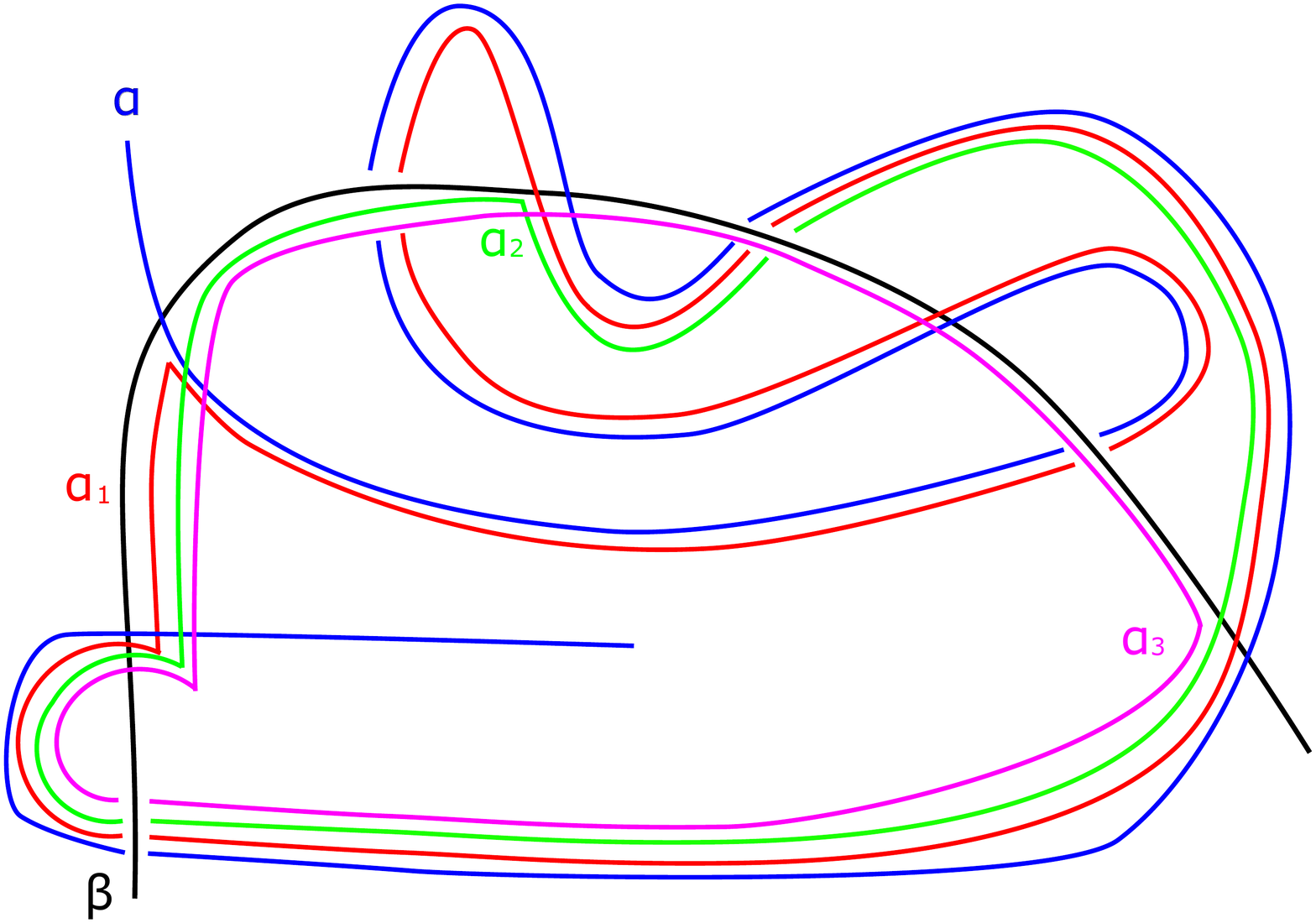}}
\vspace{0cm}
\caption{Extend the bicorn curve $\alpha_1$ along $\beta$ to obtain bicorn curves $\alpha_2$ and $\alpha_3$.}
\label{bicorn_extension}
\end{figure}

Continue in this way, one will be able to construct a sequence of bicorn curves $\alpha=\alpha_0, \alpha_1, \cdots, \alpha_n$, where the adjacent curves $\alpha_i$, $\alpha_{i+1}$ intersect at most once. Since the intersection number of $\alpha$ and $\beta$ is finite, the sequence must terminate at $\beta$, that is, $\alpha_n=\beta$. 
The sequence of bicorn curves constructed above is called a \emph{bicorn path}. We will use $B(\alpha, \beta)$ to denote one bicorn path between $\alpha$ and $\beta$, and all the bicorn paths are denoted by $\mathcal{B}(\alpha, \beta)$. Note that a bicorn path is not a real path, as the adjacent curves are not necessarily disjoint. 

The bicorn paths $\mathcal{B}(\alpha, \beta)$ have uniformly bounded Hausdorff distance to the geodesics between $\alpha$ and $\beta$. 
 \begin{Prop}
 \label{14-bounded}
 Let $\Gamma$ be a geodesic connecting two curves $\alpha$ and $\beta$, then a bicorn path $B(\alpha, \beta)$ between $\alpha$ and $\beta$ stay in the 14-neighborhood of $\Gamma$ in the curve graph. 
\end{Prop} 
In \cite{CJM}, a proof was given by Chang, Menasco and the author for the curves $\alpha$ and $\beta$ with coherent intersection.  The proof proceeds without any change for the general cases.

%%%%%%%%%%%%%%%%%%%%%%%%%%%%%%%%%%%%%%%%%%%%%%%%%%%%%%%%%%%%%%%%%%%%%%%%%%%%%%%%%%%%%%%%%%%%%%%%%%%%%%%%%%%%%%%%%%%%%%%%%%%%%%%%%%%%%%%%%%

\section{The proof}
\label{sec: proof}

In this section, we will utilize the bicorn curves to prove the Theorem \ref{small bounds} for the closed oriented surfaces $X$ with genus $ \geq 2$. With the uniformly bounded Hausdorff distance of the bicorn paths and 1-slimness of the bicorn curve triangles, the proof follows immediately from Webb's strategy to the proof of the Theorem 3.2 in \cite{Webb}. To start off, let us deal with a particular case when the geodesic is away from the curves in the subsurface. 

\begin{Lem}
\label{18}
Suppose that $X$ is a closed oriented surface with  genus $\geq 2$ and $Y$ is an essential subsurface of $X$ with $\xi(Y) \neq 0$. Let $\Gamma$ be a geodesic in $\mathcal{C}(X)$ such that each $x \in \Gamma$ cuts $Y$, and $\alpha$ be an essential boundary component of the non-annular subsurface $Y$ or the core of $Y$ if it is annular. If $d_{\mathcal{C}(X)}(x, \alpha) \geq 18$, for any vertex $x \in \Gamma$, then $d_{\mathcal{AC}(Y)}(\Gamma) \leq 3$. 
\end{Lem}

\begin{proof}
\begin{figure}[ht]
\scalebox{.2}{\includegraphics[origin=c]{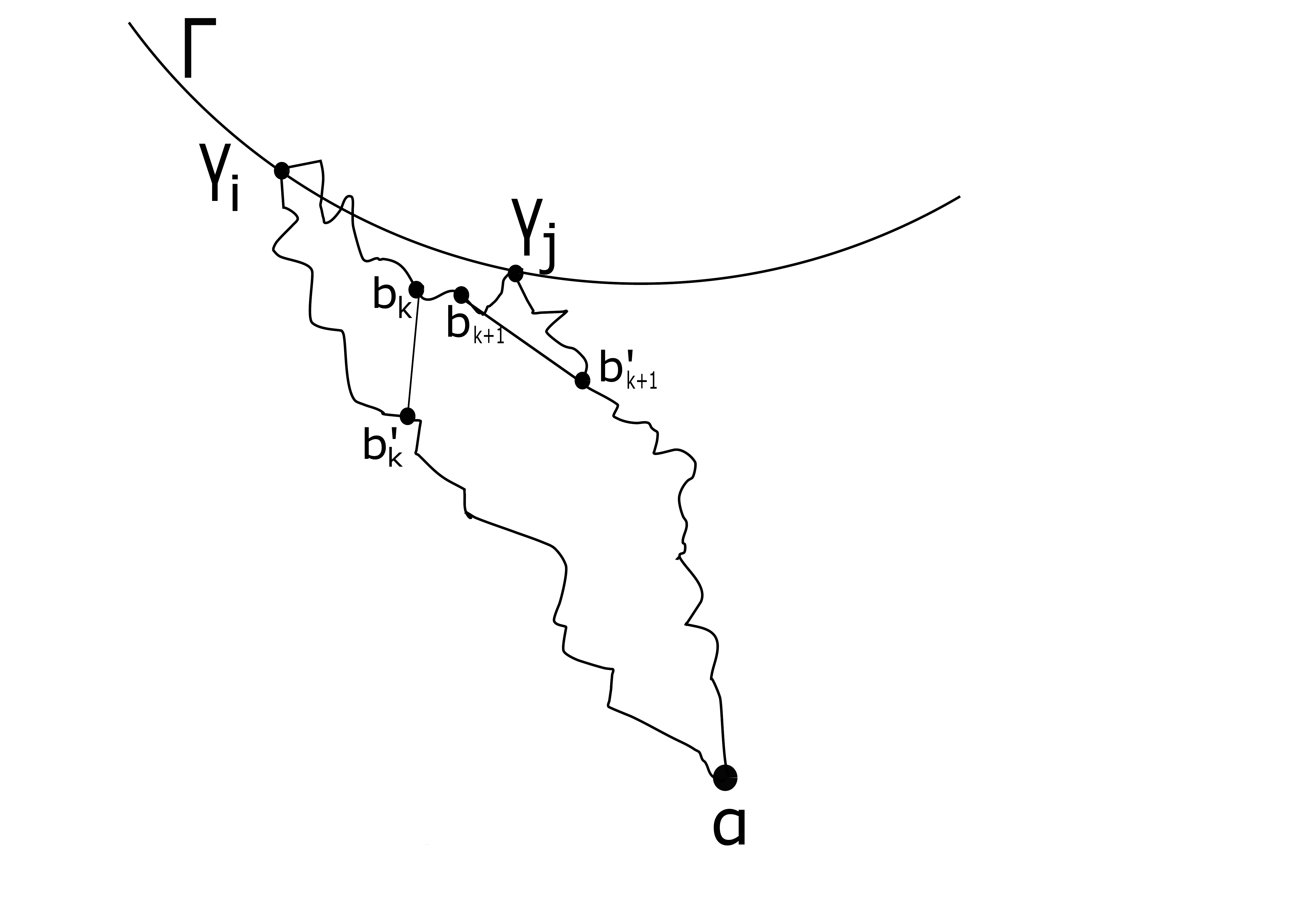}}
\caption{The bicorn paths between $\gamma_i$, $\gamma_j$ and $\alpha$ form a bicorn curve triangle. $b_k$ and $b_{k+1}$ are two adjacent bicorn curves in $B(\gamma_i, \gamma_j)$, there exists $b'_k \in B(\gamma_i, \alpha)$ and  $b'_{k+1} \in B(\gamma_j, \alpha)$ due to the 1-slimness of the bicorn curve triangle. }
\label{18-geodesics}
\end{figure}

Let $\gamma_i$ and $\gamma_j$ be any two distinct vertices in the geodesic $\Gamma$, we need to show that $d_{\mathcal{AC}(Y)}(\gamma_i, \gamma_j) \leq 3$. 
By the Proposition \ref{14-bounded}, each bicorn curve of $B(\gamma_i, \gamma_j)$ stays in the 14-neighborhood of the geodesic segment $[\gamma_i, \gamma_j] \subset \Gamma$. Since $d_{\mathcal{C}(X)}(x, \alpha) \geq 18$ for any $x \in \Gamma$, then $d_{\mathcal{C}(X)}(\gamma, \alpha) \geq 4$ for any bicorn curve $\gamma \in B(\gamma_i, \gamma_j)$. The minimal intersection number of a filling pair on a closed oriented surface $X$ with genus $\geq 2$ is 4, so $|\gamma \cap \alpha| \geq 4$.

Suppose that $\gamma_i$, $\gamma_j$ and $\alpha$ are in pairwise minimal intersection without any triple points. For any bicorn curve $\gamma = a \cup b$, where $a$ is $\gamma_i$-arc and $b$ is $\gamma_j$-arc, we consider the intersection points of $\alpha$ with the $a$ and $b$. One can take a minimal arc $c$ of $\alpha$ intersecting $a$ only at the two endpoints and $c$ intersects $b$ at most once (or intersecting $b$ only at the two endpoints and intersects $a$ at most once).  This is the key observation to prove the Lemma 2.6 \cite{PS} that the bicorn curve triangles are 1-slim, see the Figure \ref{1-slim} for an illustration. Recall that $|\gamma \cap \alpha| = |(a \cup b) \cap \alpha| \geq 4$,  then either $|a \cap \alpha| \geq 2$ or $|b \cap \alpha| \geq 2$, so the curve surgery is allowed.

\begin{figure}[ht]
\scalebox{.2}{\includegraphics[origin=c]{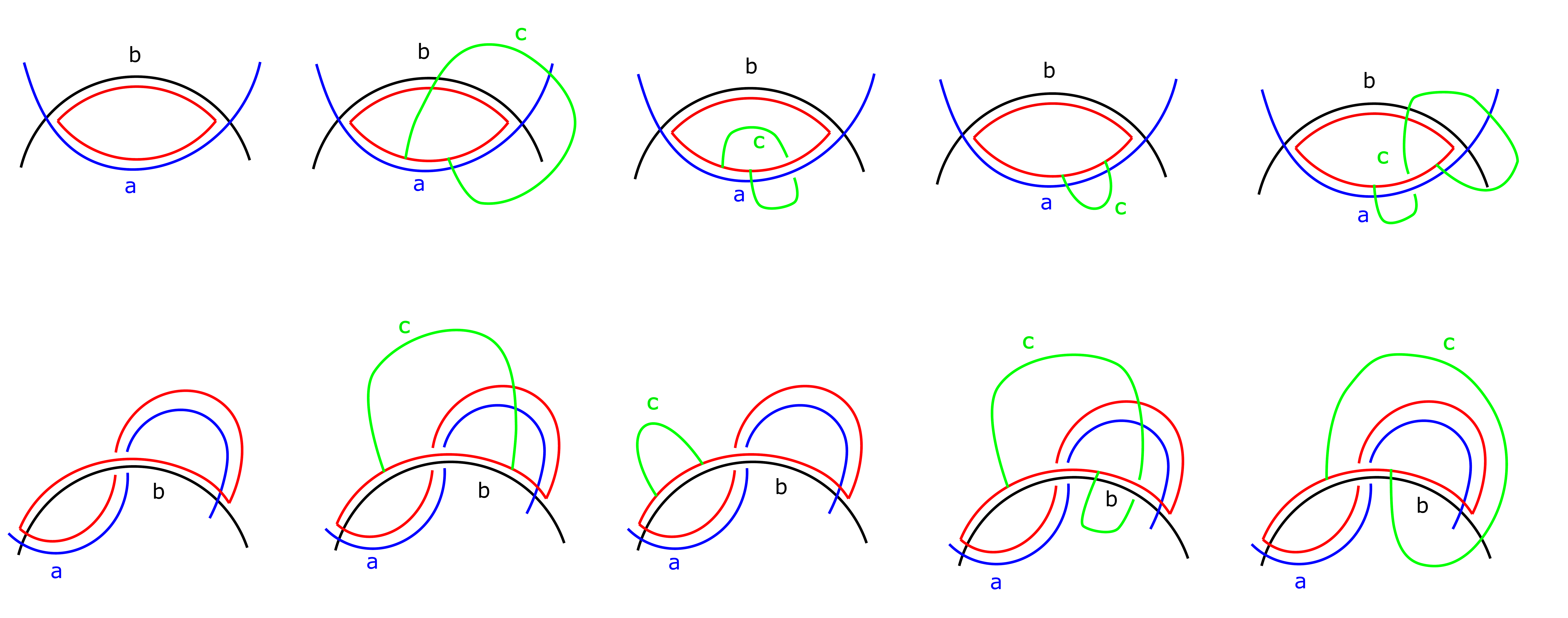}}
\caption{ Either the $\gamma_i$-arc $a$ intersects $\alpha$ at least twice or the $\gamma_j$-arc $b$ intersects $\alpha$ at least twice. The minimal arc $c \subset \alpha$ intersects $a$ (or $b$) at the two endpoints (intersection number $\leq$ 1) while intersecting the $b$ (or $a$) at most once.}
\label{1-slim}
\end{figure}

It follows that either one can construct a bicorn curve in $B(\gamma_i, \alpha)$ from $a$ and $c$ or a bicorn curve in $B(\gamma_j, \alpha)$ from $b$ and $c$. Since $\gamma_i \in B(\gamma_i, \gamma_j)$ and $\alpha$ can only produce bicorn curves in $B(\gamma_i, \alpha)$, and $\gamma_j \in B(\gamma_i, \gamma_j)$ and $\alpha$ can only produce bicorn curves in $B(\gamma_j, \alpha)$, one can conclude that there must be two adjacent bicorn curves $b_k, b_{k+1}$ (possibly same) in $B(\gamma_i, \gamma_j)$ such that $b_k$ will produce a bicorn curve $b'_k \in B(\gamma_i, \alpha)$ and $b_{k+1}$ will produce a bicorn curve $b'_{k+1} \in B(\gamma_j, \alpha)$, as illustrated in the Figure \ref{18-geodesics}.

By the construction of the curves $b_k, b'_k, b_{k+1}, b'_{k+1}$, one can find one subarc from $\pi_A(\gamma_i)$ and another subarc from $\pi_A(\gamma_j)$ that are disjoint from each other. More precisely, such subarcs can be chosen from $\pi_A(\delta_i)$ and $\pi_A(\delta_j)$, where $\delta_i$ belongs to the $\gamma_i$-arc of $b_k$ that is bounded by the endpoints of one arc $c \subset \alpha$,  and $\delta_j$ belongs to the $\gamma_j$-arc of $b_{k+1}$ that is bounded by the endpoints of another arc $c \subset \alpha $. One issue is that the $\pi_A(b'_k)$ and $\pi_A(b'_{k+1})$ might be empty. It only occurs when the arcs bounded by the endpoints of arc $c \subset \alpha$ does not have any intersection with $\alpha$ in the interior of the arcs and the two intersection points have opposite orientations. It implies that the $b'_k$ (or $b'_{k+1}$) is the bicorn curve is disjoint from $\alpha$. Then,
$$d_{\mathcal{C}(X)}(b_k, \alpha) \leq d_{\mathcal{C}(X)}(b_k, b'_k) + d_{\mathcal{C}(X)}(b'_k, \alpha) \leq 2 + 1 = 3,$$
which contradicts to $d_{\mathcal{C}(X)}(b_k, \alpha) \geq 4$. With the 1-Lipschitz property, one will obtain 
$$d_{\mathcal{AC}(Y)}(\gamma_i, \gamma_j) \leq 1 + 1 + 1 = 3.$$
\end{proof}

\begin{Rem}
The bound in the Lemma \ref{18} is comparable to the bound 4 in the motivating examples by Masur and Minsky \cite{MM1}. For the non-annular subsurfaces, the bound is same as the bound given by Webb using the unicorn arcs in the Theorem 4.1.7 \cite{Webb2}, while it is slightly better for annular subsurfaces. 
\end{Rem}

\begin{proof}[Proof of the Theorem \ref{small bounds}]
Let $\alpha$ be an essential boundary component of the non-annular subsurface $Y$ and the core of $Y$ if it is annular.  Let $I = N_{18}(\alpha) \cap \Gamma$ be the intersection (possibly empty) of the 18-neighborhood of the $\alpha$ and the geodesic $\Gamma$, then there exists a geodesic segment $g$ of length at most 36 with $I \subset g \subset \Gamma$. Denote $\gamma'_i$ and $\gamma'_j$ as the two ends of the geodesic segment $g$. Suppose at least one of $\gamma_i$ and $\gamma_j$ is not in $g$. Otherwise, since each vertex of $g$ cuts $Y$, then $d_{\mathcal{AC}(Y)}(\gamma_i, \gamma_j) \leq 36$ follows from the Lemma \ref{arc bounds}.

\begin{figure}[ht]
\scalebox{.2}{\includegraphics[origin=c]{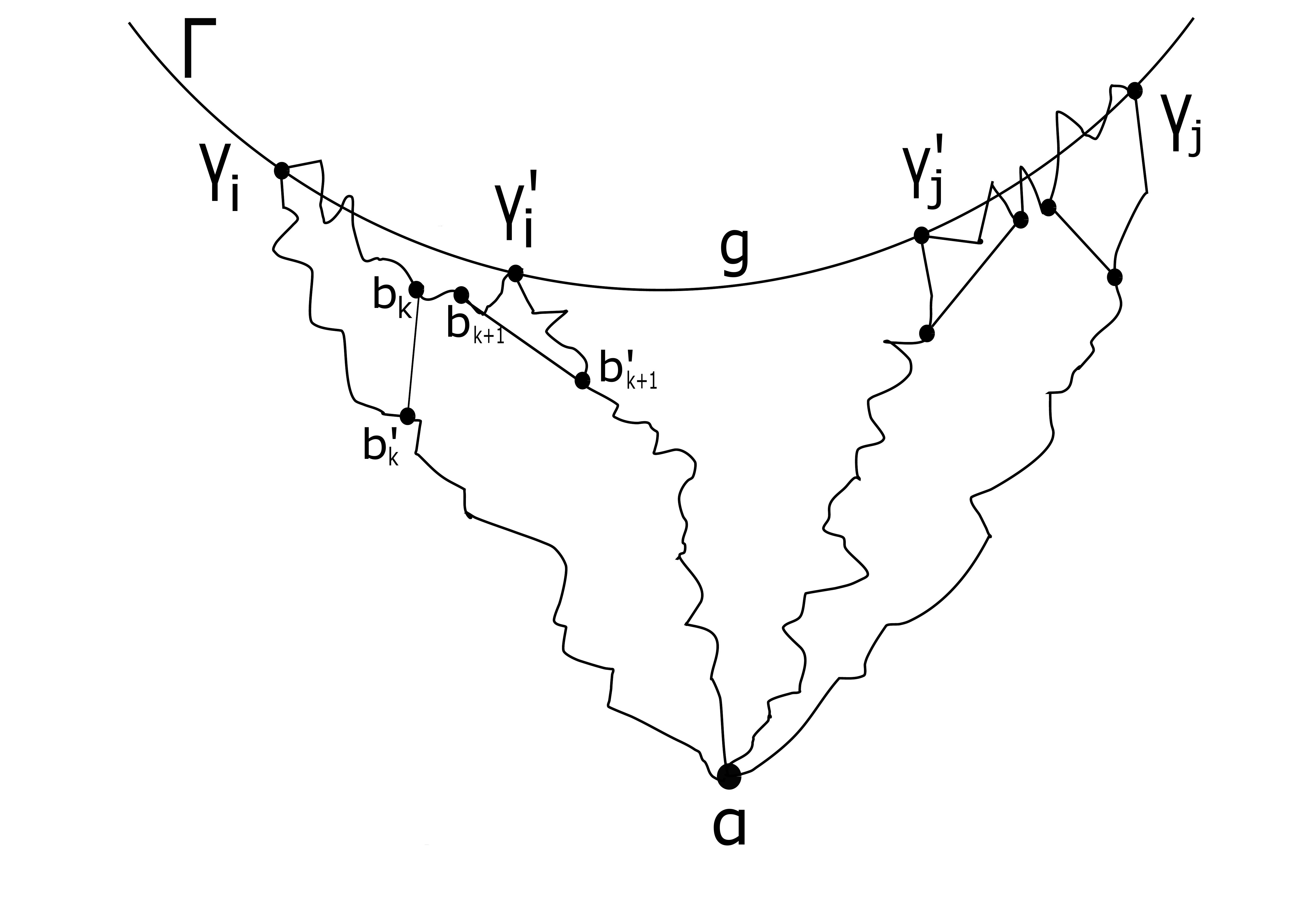}}
\caption{The geodesic $\Gamma$ is divided into three pieces. The middle geodesic segment $g = [\gamma'_i, \gamma'_j]$ (possibly empty) contains the vertices of $\Gamma$ within the 18-neighborhood of $\alpha$ in the $\mathcal{C}(X)$.  }
\label{geodesics}
\end{figure}

Assume $\gamma_i$, $\gamma_j$ or both are not in the geodesic segment $g \subset \Gamma$, see the Figure \ref{geodesics}. Using the Lemma \ref{arc bounds} and the Lemma \ref{18}, one will have 
$$d_{\mathcal{AC}(Y)}(\gamma_i, \gamma_j) \leq d_{\mathcal{AC}(Y)}(\gamma_i, \gamma'_i) + d_{\mathcal{AC}(Y)}(\gamma'_i, \gamma'_j) + d_{\mathcal{AC}(Y)}(\gamma'_j, \gamma_j) \leq  3 + 1 + 36 + 1 + 3= 44.$$
\end{proof}

\begin{Rem}
The strategy can be applied to the surfaces with boundary, because the bicorn curves can be defined in the same manner as long as the bicorn curves are essential. Restricting to the nonseparating curves, the bicorn curves traingles have been used to prove the uniform hyperbilicity of the nonseparating curve graphs by Rasmussen \cite{Ra2}. Combined with the uniformly bounded Hausdorff distance of the bicorn paths on the surfaces with boundary (Corollary 2.17, Rasmussen \cite{Ra1}), one will be able to prove a similar result, possibly with larger bounds.
\end{Rem}

\end{document}